\documentclass[a4paper,10pt]{amsart}

\usepackage{amssymb,a4wide}
\usepackage{amsthm,mathtools}
\usepackage{amsmath}
\usepackage{amsbsy}
\usepackage{graphicx}
\usepackage[all]{xy}
\usepackage{bm}
\usepackage{hyperref}
\usepackage{tikz}
\usepackage[ruled,linesnumbered]{algorithm2e}
\usetikzlibrary{shapes,arrows,shadows,positioning,decorations.pathreplacing}
\usetikzlibrary{backgrounds}
 \usetikzlibrary{calc}
\newtheorem{thm}{Theorem}

\newtheorem{corollary}{Corollary}
\newtheorem{lemma}{Lemma}

\title[]{Randomized Near Neighbor Graphs, Giant Components and Applications in Data Science}

\author[]{George C. Linderman}
\address{Program in Applied Mathematics, Yale University}
\email{george.linderman@yale.edu}

\author[]{Gal Mishne}
\address{Program in Applied Mathematics, Yale University}
\email{gal.mishne@yale.edu}

\author[]{Yuval Kluger}
\address{Department of Pathology and Program in Applied Mathematics, Yale University}
\email{yuval.kluger@yale.edu}

\author[]{Stefan Steinerberger}
\address{Department of Mathematics, Yale University}
\email{stefan.steinerberger@yale.edu}

\begin{document}
\begin{abstract} If we pick $n$ random points uniformly in $[0,1]^d$ and connect each point to its $k-$nearest neighbors, then it is well known that there exists a giant connected component with high probability. We prove that in $[0,1]^d$ it suffices to connect every point to $ c_{d,1} \log{\log{n}}$ points
chosen randomly among its $ c_{d,2} \log{n}-$nearest neighbors to ensure a giant component of size $n - o(n)$ with high probability. This construction yields a much sparser random graph with $\sim n \log\log{n}$ instead of $\sim n \log{n}$ edges that has comparable connectivity properties. This result has nontrivial implications for problems in data science where an affinity matrix is constructed: instead of picking the $k-$nearest neighbors,
one can often pick $k'  \ll k$ random points out of the $k-$nearest neighbors without sacrificing efficiency. This can massively simplify and accelerate computation, we illustrate this with several numerical examples.
\end{abstract}

\maketitle

\section{Introduction and Main Results}
\subsection{Introduction.}  The following problem is classical (we refer to the book of Penrose \cite{penrose} and references therein).
\begin{quote}
Suppose $n$ points are randomly chosen in $[0,1]^2$ and we connect every point to its $k-$nearest neighbors, what is the likelihood of obtaining a connected graph? 
\end{quote}
It is not very difficult to see that $k \sim \log{n}$ is the right order of magnitude. Arguments for both directions are sketched  in the first section of a paper by Balister, Bollob\'{a}s, Sarkar \& Walters \cite{bal}.
Establishing precise results is more challenging; the same paper shows that $k \leq 0.304 \log{n}$ leads to a disconnected graph and $k \geq 0.514 \log{n}$ leads to a connected graph with probabilities going to 1 as $n \rightarrow \infty$. We refer to \cite{bal2, bal3, falvas, walters, xue} for other recent developments.

\begin{center}
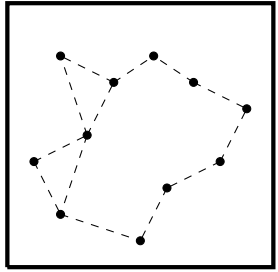
\begin{figure}[h!]
\begin{tikzpicture}[scale=3.5]
\draw [ultra thick] (0,0) -- (0,1) -- (1,1) -- (1,0) -- (0,0);
\filldraw (0.1, 0.4) circle (0.015cm);
\filldraw (0.2, 0.2) circle (0.015cm);
\filldraw (0.2, 0.8) circle (0.015cm);
\filldraw (0.3, 0.5) circle (0.015cm);
\filldraw (0.4, 0.7) circle (0.015cm);
\filldraw (0.5, 0.1) circle (0.015cm);
\filldraw (0.55, 0.8) circle (0.015cm);
\filldraw (0.8, 0.4) circle (0.015cm);
\filldraw (0.7, 0.7) circle (0.015cm);
\filldraw (0.6, 0.3) circle (0.015cm);
\filldraw (0.9, 0.6) circle (0.015cm);
\draw [dashed] (0.1, 0.4) -- (0.3, 0.5);
\draw [dashed] (0.1, 0.4) -- (0.2, 0.2);
\draw [dashed] (0.3, 0.5) -- (0.2, 0.2);
\draw [dashed] (0.3, 0.5) -- (0.4, 0.7);
\draw [dashed] (0.4, 0.7) -- (0.2, 0.8);
\draw [dashed] (0.4, 0.7) -- (0.55, 0.8);
\draw [dashed] (0.3, 0.5) -- (0.2, 0.8);
\draw [dashed] (0.7, 0.7) -- (0.55, 0.8);
\draw [dashed] (0.7, 0.7) -- (0.9, 0.6);
\draw [dashed] (0.8, 0.4) -- (0.9, 0.6);
\draw [dashed] (0.8, 0.4) -- (0.6, 0.3);
\draw [dashed] (0.6, 0.3) -- (0.5, 0.1);
\draw [dashed] (0.2, 0.2) -- (0.5, 0.1);
\end{tikzpicture}
\caption{Random points, every point is connected to its $2-$nearest neighbors.}
\end{figure}
\end{center}

We contrast this problem with one that is encountered on a daily basis in applications.
\begin{quote}
 Suppose $n$ points are randomly sampled from a set with some geometric structure (say, a submanifold in high dimensions); how should one create edges between these vertices to best reflect the underlying geometric
structure?  
\end{quote}
This is an absolutely fundamental problem in data science: data is usually
represented as points in high dimensions and for many applications one creates
an \textit{affinity matrix} that may be considered an estimate on how `close'
two elements in the data set are; equivalently, this corresponds to building a
weighted graph with data points as vertices.  Taking the $k-$nearest neighbors
is a standard practice in the field (see e.g. \cite{belkin, ulli, singer}) and will
undoubtedly preserve locality. Points are only connected to nearby points and
this gives rise to graphs that reflect the overall structure of the
underlying geometry.  The main point of our paper is that this approach, while
correct at the local geometric perspective,  is often not optimal for how
it is used in applications. We first discuss the main results from a purely
mathematical perspective and then explain what this implies for applications.

\subsection{$k-$nearest Neighbors.} We now explore what happens if $k$ is fixed and $n \rightarrow \infty$. More precisely, the question being treated in this section is as follows.
\begin{quote}
 Suppose $n$ points are randomly sampled from a nice (compactly supported, absolutely continuous) probability distribution and every point is connected to its $k-$nearest neighbors. What can be said about the number of connected components as $n \rightarrow \infty$? How do these connected components behave?
\end{quote}
The results cited above already imply that the arising graph is disconnected with very high likelihood. We aim to answer these questions in more detail. By a standard reduction (a consequence of everything being local, see e.g. Beardwood, Halton \& Hammersley \cite{beardwood}), it suffices to
study the case of uniformly distributed points on $[0,1]^d$.
More precisely, we will study the behavior of random graphs generated in the
following manner: a random sample from a Poisson process of intensity $n$ on
$[0,1]^d$ yields a number of uniformly distributed points (roughly $n \pm
\sqrt{n}$), and we connect each of these points to its $k-$nearest neighbors,
where $k \in \mathbb{N}$ is fixed. 

\begin{thm}[There are many connected components.] Let $X_n$ denote the number of connected components of a graph obtained from connecting point samples from a Poisson process with intensity $n$ to their $k-$nearest neighbors. There exists a constant $c_{d, k} > 0$ such that
$$ \lim_{n \rightarrow \infty}{ \frac{\mathbb{E} X_n}{n}} = c_{d, k}.$$
Moreover, the expected diameter of a connected component is $\lesssim_{d,k} n^{-\frac{1}{d}}$.
\end{thm}
 In terms of number of clusters, this is the worst possible behavior:
the number of clusters is comparable to the number of points.  The reason why
this problem is not an issue in applications is that the implicit
constant $c_{d,k}$ decays quickly in both parameters (see Table
\ref{fig:decay}). The second part of the statement is also rather striking: a
typical cluster lives essentially at the scale of nearest neighbor distances;
again, one would usually expect this to be a noticeable concern but in practice
the implicit constant in $\mathbb{E}~ \mbox{diam} \lesssim_{d,k} n^{-1/d}$ is
growing extremely rapidly in both parameters. It could be of interest to derive
some explicit bounds on the growth of these constants.  Our approach could be
used to obtain some quantitative statements but they are likely far
from the truth. 
\begin{table}[h!]  \begin{center}
\begin{tabular}{ l c c r }
  $k \setminus d$ & $2$ & 3 & 4\\
  $2$ & 0.049  & 0.013  & 0.0061 \\
  3 &  0.0021 & 0.00032  & 0.000089 \\
  4 & 0.00011 & 0.0000089 & 0.0000014 \\
\end{tabular}
\end{center}
\caption{Monte-Carlo estimates for the value of $c_{d,k}$ (e.g. $k=2$ nearest neighbors in $[0,1]^2$ yield roughly $\sim 0.049n$ clusters). Larger values are difficult to obtain via sampling because $c_{d,k}$ decays very rapidly.}
\label{fig:decay}

\end{table}

 We emphasize that this is a statement about the typical clusters and there are usually clusters that have very large diameter -- this is also what turns $k-$nearest neighbor graphs into a valuable tool in practice: usually, one obtains a giant connected component. Results in this direction were established by  Balister \& Bollobas \cite{balp} and Teng \& Yao \cite{teng}: more precisely, in dimension 2 the $11-$nearest neighbor graph percolates (it is believed that 11 can be replaced by 3, see \cite{balp}).

\subsection{Randomized Near Neighbors.}  Summarizing the previous sections, it
is clear that if we are given points in $[0,1]^d$ coming from a Poisson process
with intensity $n$, then the associated $k-$nearest neighbor graph will have
$\sim n$ connected components for $k$ fixed (as $n \rightarrow \infty$) and
will be connected with high likelihood as soon as $k \gtrsim c \log{n}$. The
main contribution of our paper is to show that there is a much sparser random
construction that has better connectivity properties -- this is
of intrinsic interest but has also a number of remarkable applications in
practice (and, indeed, was inspired by those).

\begin{thm} There exist constants $c_{d,1}, c_{d,2} >0$, depending only on the dimension, such that if we connect every one of $n$ points, i.i.d. uniformly sampled from $[0,1]^d$, 
$$ \mbox{to each of its}~ c_{d,1} \log{n}~\mbox{ nearest neighbors with likelihood}~ p = \frac{c_{d,2} \log\log{n}}{ \log{n}},$$
then the arising graph has a connected component of size $n-o(n)$ with high probability. 
\end{thm}
\begin{enumerate}
\item This allows us to create graphs on $\sim n \log{\log{n}}$ edges that have one large connected component of size proportional to $\sim n$ with high likelihood. 
\item While the difference between $\log{n}$ and $\log{\log{n}}$ may seem negligible for any practical applications, there is a sizeable improvement in the
explicit constant that can have a big effect (we refer to \S \ref{sec:num} for numerical examples).
\item The result is sharp in the sense that the graph is not going to be connected with high probability (see \S 5). In practical applications the constants scale favorably and the graphs are connected (in practice, even large $n$ are too small for asymptotic effects). We furthermore believe that another randomized construction, discussed in Section \S 5,
has conceivably the potential of yielding connected graphs without using significantly more edges; we believe this to be an interesting problem.
\end{enumerate}

\begin{center}
\begin{figure}[h!]  \label{fig:point}
\includegraphics[width=\textwidth]{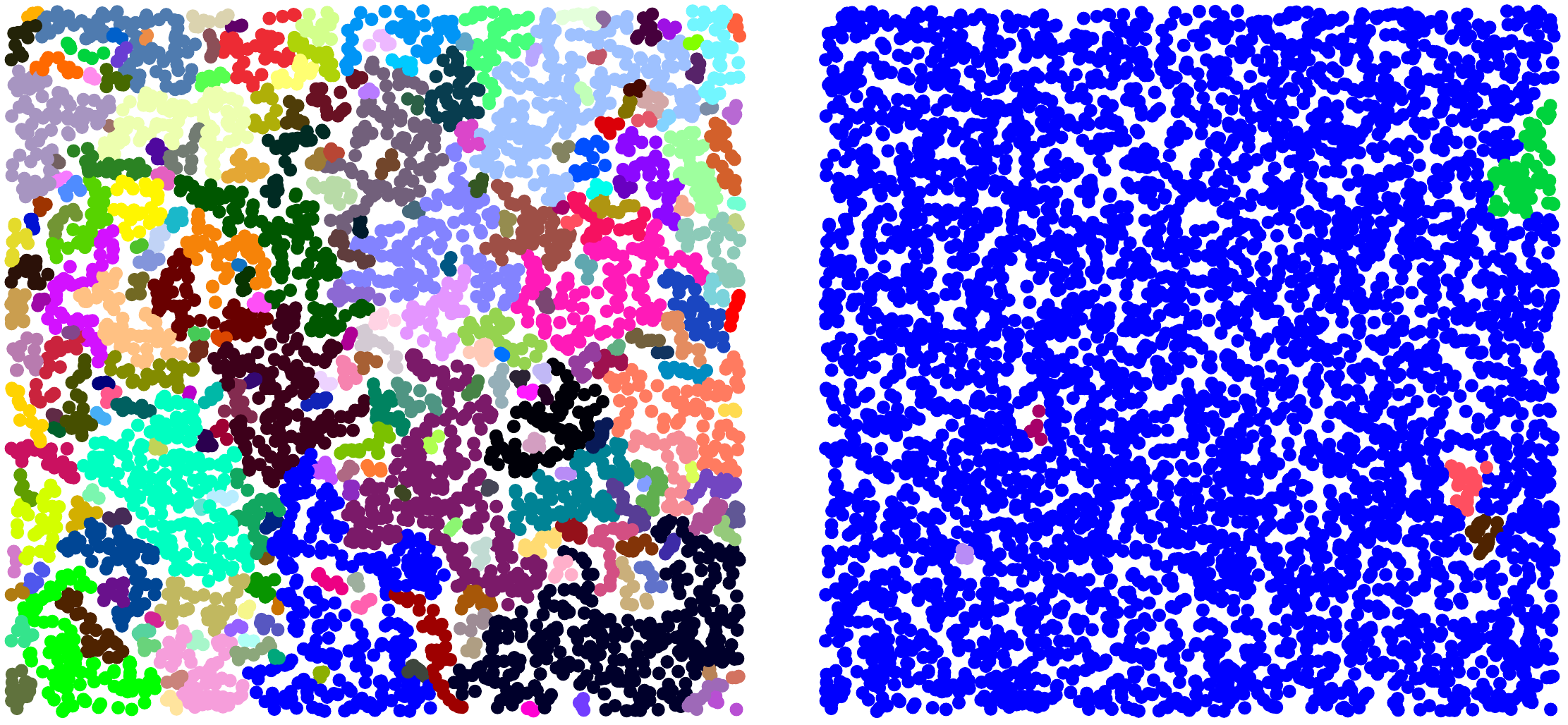}
\caption{Theorem 1 and Theorem 2 illustrated: 5000 uniformly distributed random points are connected by $2-$nearest neighbors (left) and 2 out of the $4-$nearest neighbors, randomly selected (right). Connected components are distinguished by color -- we observe a giant component on the right.}
\end{figure}
\end{center}

\subsection{The Big Picture.}
We believe this result to have many applications. Many approaches in data science require the construction of a local graph reflecting underlying structures and one often chooses
a $k-$nearest neighbor graph. If we consider the example of spectral clustering, then Maier, von Luxburg \& Hein \cite{maier} describe theoretical results regarding how this technique, applied to a $k-$nearest neighbor graph, approximates the underlying structure in the data set. Naturally, in light of the results cited above, the parameter $k$ has to grow at least like $k \gtrsim \log{n}$ for such results to become applicable. Our approach allows for a much sparser random graph that has comparable connectivity properties and several other useful properties.\\

\begin{center}
\begin{figure}[h!]  \label{fig:anomaly}
\begin{tikzpicture}[scale=4]
\draw [ultra thick] (0,0) -- (0,1) -- (1,1) -- (1,0) -- (0,0);
\filldraw (0.2, 0.2) circle (0.015cm);
\filldraw (0.2, 0.4) circle (0.015cm);
\filldraw (0.2, 0.6) circle (0.015cm);
\filldraw (0.2, 0.8) circle (0.015cm);
\filldraw (0.4, 0.2) circle (0.015cm);
\filldraw (0.4, 0.4) circle (0.015cm);
\filldraw (0.4, 0.6) circle (0.015cm);
\filldraw (0.4, 0.8) circle (0.015cm);
\filldraw (0.6, 0.2) circle (0.015cm);
\filldraw (0.6, 0.4) circle (0.015cm);
\filldraw (0.6, 0.6) circle (0.015cm);
\filldraw (0.6, 0.8) circle (0.015cm);
\filldraw (0.8, 0.2) circle (0.015cm);
\filldraw (0.8, 0.4) circle (0.015cm);
\filldraw (0.8, 0.6) circle (0.015cm);
\filldraw (0.8, 0.8) circle (0.015cm);
\draw [dashed] (0.2, 0.2) -- (0.4, 0.4);
\draw [dashed] (0.2, 0.2) -- (0.8, 0.2);
\draw [dashed] (0.2, 0.4) -- (0.8, 0.4);
\draw [dashed] (0.2, 0.6) -- (0.8, 0.6);
\draw [dashed] (0.2, 0.8) -- (0.8, 0.8);
\draw [dashed] (0.2, 0.2) -- (0.2, 0.8);
\draw [dashed] (0.4, 0.2) -- (0.4, 0.8);
\draw [dashed] (0.6, 0.2) -- (0.6, 0.8);
\draw [dashed] (0.8, 0.2) -- (0.8, 0.8);
\draw [dashed] (0.4, 0.2) -- (0.2, 0.4);
\draw [dashed] (0.6, 0.2) -- (0.4, 0.4);
\draw [dashed] (0.8, 0.2) -- (0.6, 0.4);
\draw [dashed] (0.8, 0.4) -- (0.6, 0.6);
\draw [dashed] (0.8, 0.6) -- (0.6, 0.8);
\draw [dashed] (0.8, 0.8) -- (0.6, 0.6);
\draw [dashed] (0.6, 0.8) -- (0.4, 0.6);
\draw [dashed] (0.2, 0.8) -- (0.4, 0.6);
\draw [dashed] (0.2, 0.6) -- (0.4, 0.4);
\draw [dashed] (0.4, 0.8) -- (0.6, 0.6);
\draw [dashed] (0.4, 0.4) -- (0.6, 0.6);

\draw [ultra thick] (2,0) -- (2,1) -- (3,1) -- (3,0) -- (2,0);
\filldraw (2.15, 0.15) circle (0.015cm);
\draw [dashed] (2.15, 0.15) -- (2.2, 0.35);
\draw [dashed] (2.15, 0.15) -- (2.2, 0.5);
\draw [dashed] (2.15, 0.15) -- (2.35, 0.35);
\draw [dashed] (2.15, 0.15) -- (2.3, 0.2);
\filldraw (2.2, 0.35) circle (0.015cm);
\draw [dashed] (2.2, 0.35) -- (2.3, 0.2);
\draw [dashed] (2.2, 0.35) -- (2.35, 0.35);
\draw [dashed] (2.2, 0.35) -- (2.2, 0.5);
\filldraw (2.2, 0.5) circle (0.015cm);
\draw [dashed] (2.2, 0.5) -- (2.35, 0.35);
\draw [dashed] (2.2, 0.5) -- (2.4, 0.6);
\filldraw (2.2, 0.8) circle (0.015cm);
\draw [dashed] (2.2, 0.8) -- (2.6, 0.8);
\draw [dashed] (2.2, 0.8) -- (2.4, 0.6);
\draw [dashed] (2.2, 0.8) -- (2.2, 0.5);
\filldraw (2.3, 0.2) circle (0.015cm);
\draw [dashed] (2.3, 0.2) -- (2.35, 0.35);
\draw [dashed] (2.3, 0.2) -- (2.2, 0.5);
\filldraw (2.35, 0.35) circle (0.015cm);
\draw [dashed] (2.35, 0.35) -- (2.4, 0.6);
\filldraw (2.4, 0.6) circle (0.015cm);
\draw [dashed] (2.4, 0.6) -- (2.4, 0.8);
\draw [dashed] (2.4, 0.6) -- (2.6, 0.6);
\filldraw (2.4, 0.8) circle (0.015cm);
\draw [dashed] (2.4, 0.8) -- (2.6, 0.6);
\filldraw (2.6, 0.2) circle (0.015cm);
\draw [dashed] (2.6, 0.2) -- (2.8, 0.2);
\draw [dashed] (2.6, 0.2) -- (2.6, 0.4);
\draw [dashed] (2.6, 0.2) -- (2.8, 0.4);
\draw [dashed] (2.6, 0.2) -- (2.35, 0.35);
\filldraw (2.6, 0.4) circle (0.015cm);
\draw [dashed] (2.6, 0.4) -- (2.6, 0.6);
\draw [dashed] (2.6, 0.4) -- (2.8, 0.4);
\draw [dashed] (2.6, 0.4) -- (2.8, 0.6);
\filldraw (2.6, 0.6) circle (0.015cm);
\draw [dashed] (2.6, 0.6) -- (2.6, 0.8);
\draw [dashed] (2.6, 0.6) -- (2.8, 0.6);
\filldraw (02.6, 0.8) circle (0.015cm);
\draw [dashed] (2.6, 0.8) -- (2.8, 0.8);
\draw [dashed] (2.6, 0.8) -- (2.8, 0.6);
\filldraw (02.8, 0.2) circle (0.015cm);
\draw [dashed] (2.8, 0.2) -- (2.8, 0.6);
\draw [dashed] (2.8, 0.2) -- (2.6, 0.4);
\filldraw (02.8, 0.4) circle (0.015cm);
\draw [dashed] (2.8, 0.4) -- (2.6, 0.6);
\filldraw (02.8, 0.6) circle (0.015cm);
\draw [dashed] (2.6, 0.6) -- (2.6, 0.8);
\filldraw (02.8, 0.8) circle (0.015cm);
\draw [dashed] (2.8, 0.8) -- (2.8, 0.6);
\end{tikzpicture}
\caption{Structured points, every point is connected to its $4-$nearest neighbors (some edges lie on top of each other); a slight perturbation of the points immediately creates a cluster in the $4-$nearest neighbor graph.}
\label{fig:structpoints}
\end{figure}
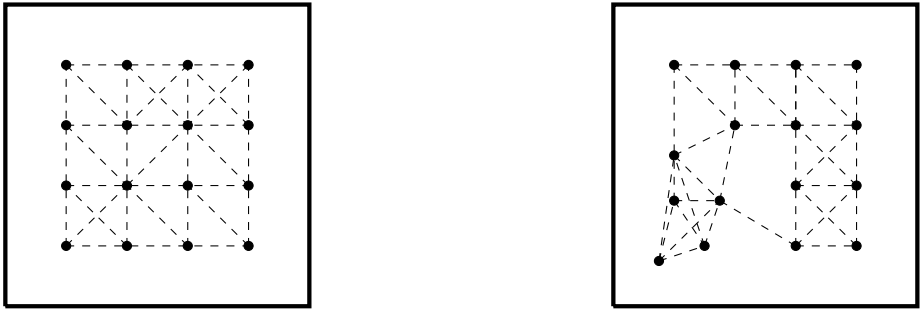
\end{center}
However, we believe that there is also a much larger secondary effect that should be extremely interesting to study: suppose we are given a set of points $\left\{x_1, \dots, x_n\right\} \subset [0,1]^2$. If
these points are well-separated then the $k-$nearest neighbor graph is an accurate representation of the underlying structure; however, even very slight inhomogeneities in the data, with some points 
being slightly closer than some other points, can have massive repercussions throughout the network (Figure \ref{fig:structpoints}). Even a slight degree of clustering will produce strongly localized
clusters in the $k-$nearest neighbor graph -- the required degree of clustering is so slight that even randomly chosen points will be subjected to it (and this is one possible interpretation of Theorem 1).\\

\textit{Smoothing at logarithmic scales.} It is easy to see that for points in $[0,1]^d$ coming from a Poisson process with intensity $n$, local clustering is happening at spatial scale $\sim (c \log{(n)}/n)^{1/d}$. The number of
points contained in a ball $B$ with volume $|B| \sim c \log{(n)}/n$ is given by a Poisson distribution with parameter $\lambda \sim c \log{n}$ and satisfies 
$$ \mathbb{P}\left(\mbox{$B$ contains less than}~\ell~\mbox{points}\right) \sim \frac{(c\log{n})^\ell}{\ell!}  \frac{1}{n^c} \lesssim \frac{1}{n^{c-\varepsilon}}.$$
This likelihood is actually quite small, which means that it is quite unlikely to find isolated clusters at that scale. In particular, an algorithm as proposed in Theorem 2 that picks random elements at that scale, will then destroy the nonlinear
concentration effect in the $k-$nearest neighbor construction induced by local irregularities of uniform distribution.
We believe this to be a general principle that should have many applications.
\begin{quote}
When dealing with inhomogeneous data, it can be useful to select $K \gg k$ nearest neighbors and then subsample $k$ random elements. Here, $K$ should be chosen at such a scale that
localized clustering effects disappear. 
\end{quote}

An important assumption here is that clusters at local scales are not intrinsic but induced by unavoidable sampling irregularities. We also emphasize that we believe the best way to implement
this principle in practice and in applications to be very interesting and far from resolved.

\section{Applications}

\subsection{Implications for Spectral Clustering.}
The first step in spectral clustering of points
$\{x_1,...,x_n\}\subset \mathbb{R}^d$ into $p$ clusters involves computation of a kernel
$w(x_i,x_j)$ for all pairs of points, resulting in an $n \times
n$ matrix $W$.  The kernel is a measure of affinity 
and is commonly chosen to be a Gaussian with bandwidth $\sigma$,
$$w(x_i,x_j) = \exp\left(-\|x_i-x_j\|^2/\sigma^2\right).$$ 
$W$ defines a graph with $n$ nodes where the weight of the edge between $x_i$
and $x_j$ is $w(x_i,x_j)$.  The Graph Laplacian $L$ is defined as:
$$L = D - W$$
where $D$ is a diagonal matrix with the sum of each row on the diagonal.  Following \cite{vonLuxburg}, the Graph Laplacian can be normalized symmetrically
$$L_{sym} = D^{-1/2}L D^{1/2} = I - D^{-1/2}WD^{-1/2},$$
giving a normalized Graph Laplacian. The eigenvectors $\{v_1,...,v_p \}$ corresponding to the $p$ smallest
eigenvalues of $L_{sym}$ are then calculated and concatenated into the $n
\times p$ matrix $V$. The rows of $V$ are normalized to have unit norm,
and its rows are then clustered using the $k-$means algorithm into $p$ clusters.
Crucially, the multiplicity of the eigenvalue of $0$ of $L_{sym}$ equals the
number of connected components, and the eigenvectors corresponding to the $0$th
eigenvalue are piecewise constant on each connected component. In the case of
$p$ well-separated clusters, each connected component corresponds to a cluster,
and the first $p$ eigenvectors contain the information necessary to separate
the clusters.  

\begin{center}
\begin{figure}[h!]
\includegraphics[width=0.8\textwidth]{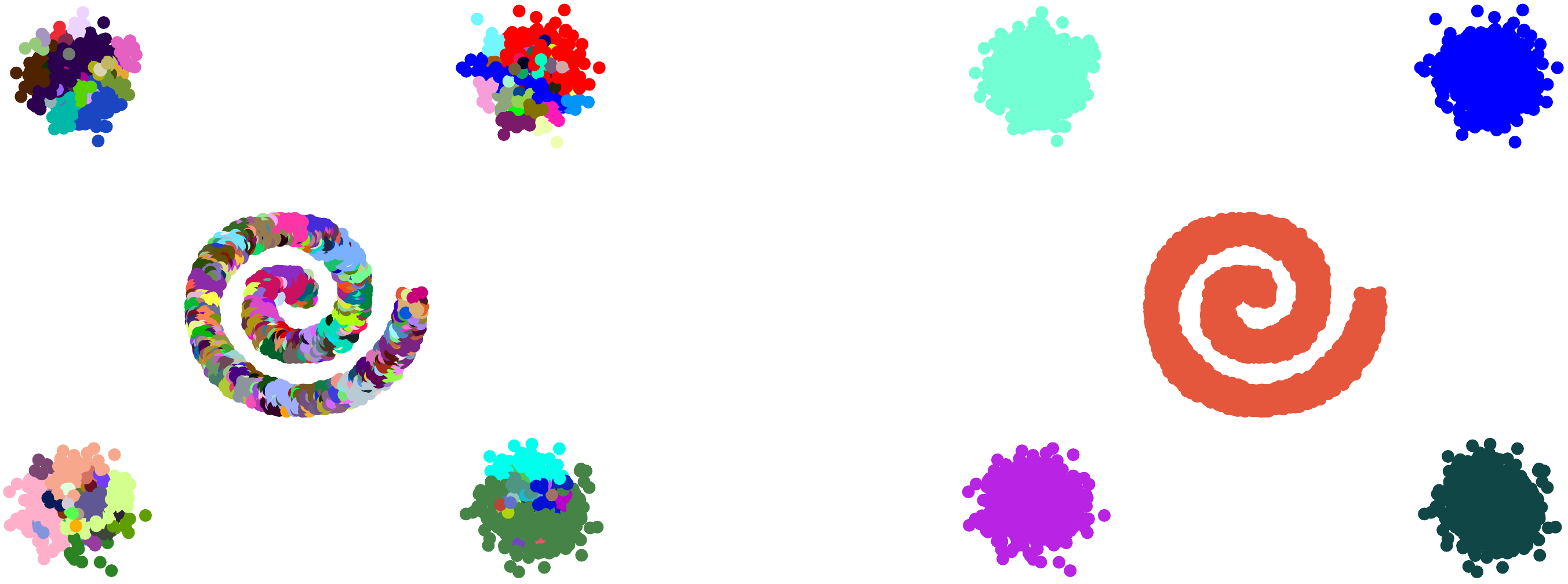}
\caption{16000 points arranged in 4 clusters and a spiral. We compare the effect of connecting every point to its $2-$nearest neighbors (left) and connecting every point to 2 randomly chosen out of
its $7-$nearest neighbors (right). Connected components are colored, the graph on the left has $\sim 700$ connected components; the graph on the right consists of the actual 5 clusters.}
\label{fig:spiral}
\end{figure}
\end{center}
However, the computational complexity and memory storage requirements of $W$
scale quadratically with $n$, typically making computation intractable when $n$
exceeds tens of thousands of points. Notably, as the distance between any two points $x$ and
$y$ increases, $w(x,y)$ decays exponentially. That is, for all $x,y$ that
are not close (relative to $\sigma$), $w(x,y)$ is negligibly small. Therefore, $W$ can be approximated by a sparse matrix $W'$, where $W_{ij}' =
w(x_i,x_j)$ if $x_j$ is among $x_i$'s $k-$nearest neighbors or $x_i$ is
among $x_j$'s nearest neighbors, and $0$ otherwise. Fast algorithms have been
developed to find approximate nearest neighbors (e.g.  \cite{jones}), allowing
for efficient computation of $W'$, and Lanczos methods can be used to
efficiently compute its eigenvectors.
When the number of neighbors, $k$, is chosen sufficiently large, $W'$ is a
sufficiently accurate approximation of $W$ for most applications.  However,
when $k$ cannot be chosen large enough (e.g.  due to memory limitations when
$n$ is on the order of millions of points), the connectivity of the $k-$nearest
neighbor graph represented by $W'$ can be highly sensitive to noise,
leading $W'$ to overfit the data and poorly model the overall structure.  In
the extreme case, $W'$ can lead to a large number of disconnected components
within each cluster, such that the smallest eigenvectors correspond to each of
these disconnected components and not the true structure of the data. On the
other hand, choosing a random $k-$sized subset of $K-$nearest neighbors, for
$K>k$, results in a graph with the same number of edges but which is much more
likely to be connected within each cluster, and hence, allow for spectral
clustering (Figure \ref{fig:spiral}).  The latter strategy is a more effective
``allocation'' of the $k$ edges, in the resource limited setting.

\subsection{Numerical Results.}   \label{sec:num}
We demonstrate the usefulness of this approach on the MNIST8M dataset generated
by InfiMNIST \cite{mnist8m}, which provides an unlimited supply of handwritten
digits derived from MNIST using random translations and permutations. For
simplicity of visualization, we chose digits $3,6,7$, resulting in a dataset of
$n=2,472,390$ in $d=784$ dimensional space.  We then computed the first ten
principal components (PCs) using randomized principal component analysis
\cite{li2017algorithm} and performed subsequent analysis on the PCs. Let
$L^{k}_{\text{sym}}$ denote the symmetrized Laplacian of the graph formed by
connecting each point to its $k-$nearest neighbors, with each edge weighted using the
Gaussian kernel from above with an adaptive bandwidth $\sigma_i$ equal to the point's 
distance to its $k$th neighbor. Similarly, let $L^{k,K}_{\text{sym}}$ refer to
the Laplacian of the graph formed by connecting each point to a $k-$sized
subset of its $K-$nearest neighbors, where each edge is weighted with a Gaussian of
bandwidth equal to the squared distance to its $K$th nearest neighbor. We then used the Lanczos iterations as implemented in MATLAB's `eigs' function to
compute the first three eigenvectors of $L^{30}_{\text{sym}}$,
$L^{50}_{\text{sym}}$, and $L^{2,100}_{\text{sym}}$ which we plot in Figure
\ref{fig:mnist}.

\begin{center}
\begin{figure}[h!]  
\begin{tikzpicture}
\node (0,0) {\includegraphics[width=\textwidth]{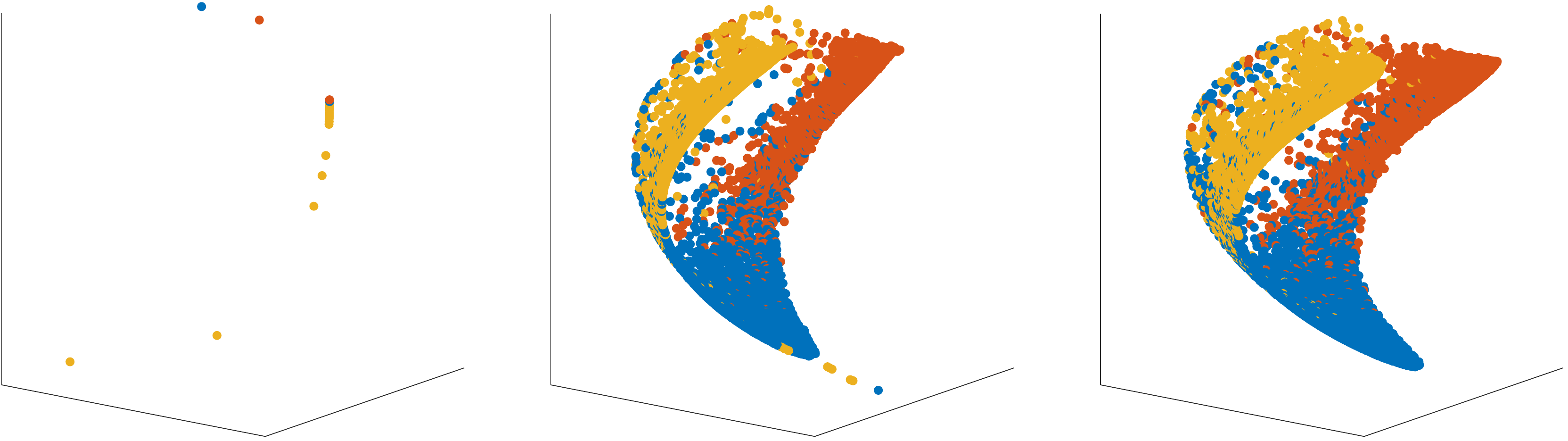}};
\node at (-5.3,-2.5) {$30-$nearest neighbors};
\node at (-5.3,-2.9) {48 minutes};
\node at (-0,-2.5) {$50-$nearest neighbors};
\node at (-0,-2.9) {16 minutes};
\node at (5.3,-2.5) {2 out of $100-$nearest neighbors};
\node at (5.3,-2.9) {4 minutes};
\end{tikzpicture}
\caption{Eigenvectors of sparse Graph Laplacian of three digits in the
Infinite-MNIST data set. Connecting to $k-$nearest neighbors when $k$ is too
small leads to catastrophic results (left). Connecting to 2 randomly chosen out
of the $100-$nearest neighbors points is comparable to connecting to the 50
nearest neighbors, but it requires fewer edges and computing the top
three eigenvectors is much faster (time under each plot).}
\label{fig:mnist}
\end{figure}
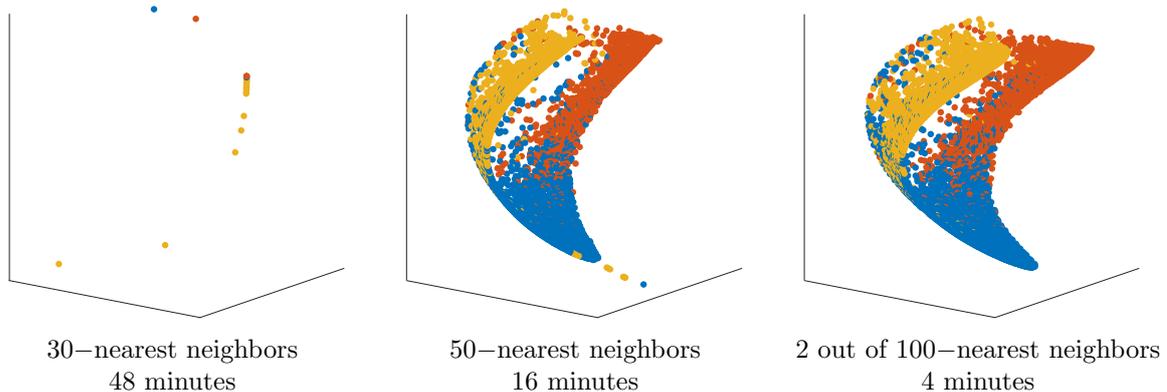
\end{center}

The first three eigenvectors of $L^{30}_{\text{sym}}$ do not separate the
digits, nor do they reveal the underlying manifold on which the digits lie.
Increasing the number of nearest neighbors in $L^{50}_{\text{sym}}$
provides a meaningful embedding. Remarkably, the same quality embedding can be
obtained with $L^{2,100}_{\text{sym}}$, despite it being a much sparser graph.
Furthermore, computing the first three eigenvectors of $L^{2,100}_{\text{sym}}$ took only 4
minutes, as compared to 48 and 16 minutes for $L^{30}_{\text{sym}}$ and 
$L^{50}_{\text{sym}}$.

\subsection{Potential Implementation}   \label{sec:implementation}
In order to apply this method to large datasets on resource-limited machines,
an efficient algorithm for finding a $k-$sized subset for the $K-$nearest
neighbors of each point is needed. For the above experiments we simply computed
all $K-$nearest neighbors and randomly subsampled, which is clearly suboptimal. Given a dataset so large that $K-$nearest neighbors cannot be
computed, how can we find $k-$sized random subsets of the $K-$nearest
neighbors for each point? Interestingly, this corresponds to an ``inaccurate'' nearest
neighbors algorithm, in that the ``near'' neighbors of each point are sought,
not actually the ``nearest.''  From this perspective, it appears an easier
problem than that of finding the nearest neighbors. We suggest a simple and fast implementation which we have found to
be empirically successful in Algorithm $\ref{simplealgo}$.

\begin{algorithm}[ht]
	\caption{Simple method for finding near neighbors}
	\label{simplealgo}
	\KwIn{Dataset $A=\{x_1,...,x_n\}\subset \mathbb{R}^d$, non-negative integers $k, K$ with $k<K$ }
	\KwOut{Matrix $M$ of size $n \times k$  where $M_{i1},...,M_{ik}$ are the indices of a random subset of $\sim K-$nearest neighbors of $x_i \in A$.  }

	Let $m = \left\lfloor\frac{kn}{K}\right\rfloor$

	Let $B \subseteq A$ be a set of $m$ points randomly selected from A.
	
	For each $x_i \in A$ find its $k$-nearest neighbors in $B$ and concatenate the indices of these points into the $i$-th row of $M$.

\end{algorithm}
On average $k$ points from $B$ will be among the $K-$nearest neighbors of any
point in $A$. As such, every point will connect to a $k-$sized subset of its $\sim K-$nearest neighbors.  Choosing a single subset of points, however,
dramatically reduces the randomness in the algorithm, and hence is not ideal.
We include it here for its simplicity and its success in our preliminary
experiments.

\subsection{Further outlook.} 
We demonstrate the application of our approach in the context of spectral
clustering but this is only one example. There are a great many other methods
of dimensionality reduction that start by constructing a graph that roughly approximates
the data, for example t-distributed Stochastic Neighborhood Embedding (t-SNE)
\cite{linderman2017clustering, maaten2008visualizing}, diffusion maps \cite{raphey} or Laplacian Eigenmaps
\cite{belkin}. Basically, this refinement could possibly be valuable for a very
wide range of algorithms that construct graph approximations out of underlying
point sets -- determining the precise conditions under which this method is
effective for which algorithm will strongly depend on the context, but needless
to say, we consider the experiments shown in this section to be extremely
encouraging.
We believe that this paper suggests many possible directions for future
research: are there other natural randomized near neighbor constructions
(we refer to \S 5 for an example)? Other questions include the behavior of
the spectrum and the induced random walk -- here we would like to briefly
point out that random graphs are natural expanders \cite{kol, mar, pin}.
This should imply several additional degrees of stability that the standard
$k-$nearest neighbor construction does not have.

\section{Proof of Theorem 1}

\subsection{A simple lower bound}
We start by showing that there exists a constant $\varepsilon_{d, k}$ such that
$$ \mathbb{E}  X_n \geq \varepsilon_{d, k} n.$$
This shows that the number of connected components grows at least linearly. 
\begin{proof} We assume that we are given a Poisson process with intensity $n$ in $[0,1]^d$.
\begin{figure}[h!]
\begin{tikzpicture}[scale=1]
\draw[thick] (0,0) circle (0.5cm);
\node [below] at (0,0) {$B$};
\draw[<->]
	(0,0) -- (0.5,0) node [above,midway]              {\small $r$};
\draw[thick] (0,0) circle (1.5cm);
\node [below] at (0,-0.7) {$A$};
\draw[<->]
	(0.5,0) -- (1.5,0) node [above,midway]              {\small $2r$};
\end{tikzpicture}
\caption{The likelihood of finding $\ell$ points in $B$ and no points in $A \setminus B$ for $|A| \sim n^{-1}$ and $|B| \ll |A|$ is uniformly bounded from below.}
\end{figure}
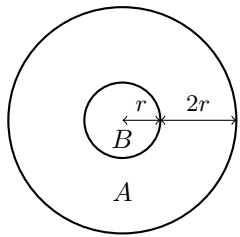

 It is possible
to place $\sim n$ balls of radius $r \sim n^{-1/d}$ in $[0,1]^d$ such that the balls with the same center and 
radius $3r$ do not intersect. (The implicit constant is related to the packing density of balls and decays rapidly in the dimension.)  The probability of finding $\ell$ points $\Omega \subset [0,1]^d$ is
$$ \mathbb{P}\left( \Omega~\mbox{contains}~\ell~\mbox{points}\right) = e^{-|\Omega|n} \frac{(|\Omega|n)^{\ell}}{\ell!}$$

This implies that the likelihood of finding $k$ points in the ball of radius $r$ and 0 points in the spherical shell obtained
from taking a ball with radius $3r$ and removing the ball of radius $r$ is given by a fixed constant independent of $n$
(because these sets all have measure $\sim n^{-1}$). This implies the result since the events in these disjoint balls are independent.
\end{proof}

\subsection{Bounding the degree of a node} \label{sec:pack}

Suppose we are given a set of points $\left\{x_1, \dots, x_n\right\} \subset \mathbb{R}^d$ and assume that every
vertex is connected to its $k-$nearest neighbors. 

\begin{quote}
\textbf{Packing Problem.} What is the maximal degree of a vertex in a $k-$nearest neighbor graph created by any set of points in $\mathbb{R}^d$ dimensions?
\end{quote}

 We will now prove the existence of a constant $c_d$, depending only on the dimension, such that the maximum
degree is bounded by $c_d k$. It is not difficult to see that this is the right order of growth in $k$: in $\mathbb{R}^d$, we can put a point in the origin and find a set of distinguished points at distance 1 from
the origin and distance 1.1 from each other. Placing $k$ points close to each of the distinguished points yields a construction of points where the degree of the point in the origin is $c_d^* k$, where $c_d^*$ 
is roughly the largest number of points one can place on the unit sphere so that each pair is at least $1-$separated.

\begin{figure}[h!]  \label{fig:three}
	\begin{tikzpicture}[scale=1]
\coordinate (X) at (0,0);
\coordinate (A) at ($(X)+(0:10mm)$  );
\coordinate (A2) at ($(X)+(10:10mm)$  );
\coordinate (B) at ($(X)+(73:10mm)$ );
\coordinate (B2) at ($(X)+(83:10mm)$ );
\coordinate (C) at ($(X)+(144:10mm)$);
\coordinate (C2) at ($(X)+(154:10mm)$);
\coordinate (D) at ($(X)+(216:10mm)$);
\coordinate (D2) at ($(X)+(226:10mm)$);
\coordinate (E) at ($(X)+(288:10mm)$);
\coordinate (E2) at ($(X)+(298:10mm)$);
\filldraw [black,fill=black](A) circle (0.05cm);
\filldraw [black,fill=black](A2) circle (0.05cm);
\draw [dashed] (A) -- (A2);
\draw [dashed] (A) -- (X);
\draw [dashed] (A2) -- (X);
\filldraw [black,fill=black](B) circle (0.05cm);
\filldraw [black,fill=black](B2) circle (0.05cm);
\draw [dashed] (B) -- (B2);
\draw [dashed] (B) -- (X);
\draw [dashed] (B2) -- (X);
\filldraw [black,fill=black](C) circle (0.05cm);
\filldraw [black,fill=black](C2) circle (0.05cm);
\draw [dashed] (C) -- (C2);
\draw [dashed] (C) -- (X);
\draw [dashed] (C2) -- (X);
\filldraw [black,fill=black](D) circle (0.05cm);
\filldraw [black,fill=black](D2) circle (0.05cm);
\draw [dashed] (D) -- (D2);
\draw [dashed] (D) -- (X);
\draw [dashed] (D2) -- (X);
\filldraw [black,fill=black](E) circle (0.05cm);
\filldraw [black,fill=black](E2) circle (0.05cm);
\draw [dashed] (E) -- (E2);
\draw [dashed] (E) -- (X);
\draw [dashed] (E2) -- (X);
\filldraw [black,fill=black](X) circle (0.05cm);
\end{tikzpicture}
\caption{A configuration where a point has $5k$ neighbors, with $k=2$.} 
\end{figure}
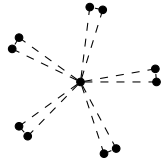

\begin{lemma} \label{lem:deg} The maximum degree of a vertex in a nearest-neighbor graph on points in $\mathbb{R}^d$ is bounded from above by $c_d k$. 
\end{lemma}

\begin{proof}
In a $k-$nearest neighbor graph any node $x$ has at least $k$ edges since it connects to its $k-$nearest neighbors. It therefore suffices to bound the number
of vertices that have $x$ among its $k-$nearest neighbors. Let now $C$ be a cone with apex in $x$ and opening angle $\alpha = \pi/3$.
\begin{figure}[h!]
\begin{tikzpicture}[scale=2]
	\foreach \angle [count=\xi, evaluate=\xi as \xx using int(\xi*12)] in {60,0,...,-270}{
		  \draw[line width=0.5pt] (\angle:0) -- (\angle:1);
	    }
\draw (0,0) circle (0.707cm);
\filldraw [black,fill=black](0,0) circle (0.03cm);
	\node [below] at (0,-0.1) {$x$};
\filldraw [black,fill=black](.4,.3) circle (0.02cm);
\filldraw [black,fill=black](.5,.2) circle (0.02cm);
\filldraw [black,fill=black](.3,.2) circle (0.02cm);
\filldraw [black,fill=black](.5,.5) circle (0.02cm);
\filldraw [black,fill=black](0.4,0.4) circle (0.02cm);
\pgfdeclarelayer{bg}    
\pgfsetlayers{bg,main} 
\begin{pgfonlayer}{bg}    
\coordinate (A) at (0,0);
\coordinate (X) at (1,0);
\coordinate (Y) at (0.5,0.86);
\path[clip] (A) -- (X) -- (Y);
\fill[white, draw=black] (A) circle (3mm);
\node at ($(A)+(30:2mm)$) {$\theta$};
\end{pgfonlayer}
\end{tikzpicture}
\caption{Six cones with angle $\theta = \pi/3$ covering $\mathbb{R}^2$. In any cone $C$, there are at most $k$ points for which
$x$ is among their $k-$nearest neighbors.} 
\end{figure}
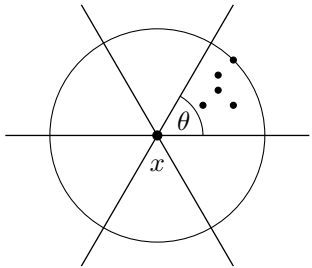
 Then, by definition,
for any two points $a,b \in C$, we have that
$$ \frac{\langle a-x,b-x \rangle}{\|a-x\| \|b-x\|} = \cos{\left(\angle(a-x, b-x)\right)} \geq \cos{\alpha} = \frac{1}{2}.$$
We will now argue that if $a$ has a bigger distance to $x$ than $b$, then $b$ is closer to $a$ than $x$. Formally,
we want to show that $\|a-x\| > \|b - x\|$ implies $\| a-b \| < \|a- x\|$. We expand the scalar product and use the inequality above to write
\begin{align*}
\| a- b\|^2 = \| (a-x) - (b-x) \|^2 &= \left\langle a-x, a-x \right\rangle - 2 \left\langle a-x, b-x \right\rangle +  \left\langle b-x, b-x \right\rangle \\
&\leq \|a - x\|^2 + \|b-x\|^2 - \|a-x\|\|b-x\| \\
&< \|a - x\|^2.
\end{align*}

Now we proceed as follows: we cover $\mathbb{R}^d$ with cones of opening angles $\alpha = \pi/3$ and apex in $x$. Then, clearly, the previous argument
implies that every such cone can contain at most $k$ vertices different from $x$ that have $x$ as one of their $k-$nearest neighbors. $c_d$ can thus be chosen
as one more than the smallest number of such cones needed to cover the space.
\end{proof}
	
\textbf{A useful Corollary.} We will use this statement in the following way: we let the random variable $X_n$ denote the number of clusters of $n$ randomly
chosen points w.r.t. some probability measure on $\mathbb{R}^n$ (we will apply it in the special case of uniform distribution on
$[0,1]^d$ but the statement itself is actually true at a greater level of generality). 

\begin{corollary} \label{cor1} The expected number of clusters cannot grow dramatically; we have
$$ \mathbb{E} X_{n+1} \leq \mathbb{E} X_n + c_d k.$$
\end{corollary}
\begin{proof}
We prove a stronger statement: for any given given set of of points $\left\{x_1, \dots, x_n \right\} \in \mathbb{R}^d$ and any $x \in \mathbb{R}^d$,
we have that the number of clusters in $\left\{x_1, \dots, x_{n}, x\right\}$ is at most $c_d k$ larger than the number of clusters in  $\left\{x_1, \dots, x_{n}\right\}$.
Adding $x$ is going to induce a localized effect in the graph: the only new edges that are being created are the $k-$nearest neighbors of $x$ that are being added
as well as changes coming from the fact that some of the points $x_1, \dots, x_n$ will now have $x$ as one of their $k-$nearest neighbors. We have already
seen in the argument above that this number is bounded by $c_d k$. This means that at most $c_d k$ of the existing edges are being removed.
Removing an edge can increase the number of clusters by at most 1 and this then implies the result.
\end{proof}

\subsection{The diameter of connected components}
The fact that most connected components are contained in a rather small region of space follows relatively quickly from Theorem 1 and the following consequence of the degree bound.

\begin{lemma} \label{lem:unif}
	Let $\left\{x_1, \dots, x_n \right\} \subset [0,1]^d$. Summing the distances over all pairs where one is a $k-$nearest neighbor of the other is bounded by
$$ \sum_{x_i, x_j ~{\tiny \mbox{knn}}}{ \|x_i - x_j\|} \lesssim_{k,d}  n^{\frac{d-1}{d}}.$$
\end{lemma}
\begin{proof} Whenever $x_j$ is a $k-$nearest neighbor, we put a ball $B(x_i, \|x_j - x_i\|)$ of radius $\|x_j - x_i\|$ around $x_i$. A simple
application of H\"older's inequality shows that
$$ \sum_{x_i, x_j ~{\tiny \mbox{knn}}}{ \|x_i - x_j\|}  \leq (kn)^{\frac{d-1}{d}} \left(  \sum_{x_i, x_j ~{\tiny \mbox{knn}}}{ \|x_i - x_j\|^d} \right)^{\frac{1}{d}}  \lesssim_{k,d} n^{\frac{d-1}{d}}  \sum_{ x_i, x_j ~{\tiny \mbox{knn}}}{\left|B(x_i, \|x_j - x_i\|) \right|}.$$ 
Lemma \ref{lem:deg} shows that each point in $[0,1]^d$ can be contained in at most $c_d k$ balls (otherwise adding a point would create a vertex with too large a degree). This implies
$$ \sum_{ x_i, x_j ~{\tiny \mbox{knn}}}{\left|B(x_i, \|x_j - x_i\|) \right|} \leq c_k 5^d \lesssim_{k,d} 1$$
and we have the desired result.
\end{proof}
We note that this result has an obvious similarity to classical deterministic upper bounds on the length of a traveling salesman path, we refer to \cite{few, steel, steint} for examples. Nonetheless,
while the statements are similar, the proof of this simple result here is quite different in style. It could be of interest to obtain some good upper bounds for this problem.

\begin{corollary}
The diameter of a typical connected component is $\lesssim_{d,k} n^{-\frac{1}{d}}.$
\end{corollary}
This follows easily from the fact that we can bound the sum of the diameters of all connected component by 
the sum over all distances of $k-$nearest neighbors. Put differently, the typical cluster is actually contained 
in a rather small region of space; however, we do emphasize that the implicit constants (especially in the estimate on the number of clusters)
are rather small and thus the implicit constant in this final diameter estimate is bound to be extremely large. This means that this phenomenon
is also not usually encountered in practice even for a moderately large number of points.
As for Lemma 2 itself, we can get much more precise results if we assume that
the points stem from a Poisson process with intensity $n$ on $[0,1]^d$.
We prove a Lemma that is fairly standard; the special cases $k=1,2$ are easy to find (see e.g. \cite{stein} and references therein); we provide 
the general case for the convenience of the reader.
\begin{lemma}
 The probability distribution function of the distance $r$ of a fixed point to its $k$th nearest neighbor in a Poisson process with intensity $n$ is
	$$f_{k,d}(r) = \frac{d n^k \omega_d^k r^{kd-1}}{(k-1)!} e^{-n \omega_d r^d},~\mbox{where} \qquad \omega_d = \frac{\pi^{\frac{d}{2}}}{\Gamma(\frac{d}{2} +1 )}.$$
\end{lemma}
\begin{proof} The proof proceeds by induction, the base case is elementary.
	We derive the cumulative distributive function and then differentiate
	it. First, recall that for Borel measurable region $B \subset
	\mathbb{R}^d$, the probability of finding $\ell$ points in $B$ is
	$$\mathbb{P}\left( B~\mbox{contains}~\ell~\mbox{points}\right) = e^{-n |B| } \frac{(n |B|)^{\ell}}{\ell !} $$
 Let $F_{k,d}(r)$ denote the probability that the $k-$nearest neighbor is at least at a distance $r$ and let, as usual, $B_r \subset \mathbb{R}^d$ denote a ball of radius $r$.
	\begin{align*}
	F_{k,d}(r) = 1 - \sum\limits_{\ell=0}^{k-1}  \mathbb{P} \left( B_r~\mbox{contains}~\ell~\mbox{points} \right)  = 1 - \left(e^{-n \omega_d r^{d}} + \sum\limits_{\ell=1}^{k-1} \frac{n^{\ell} \omega_d^{\ell} r^{\ell d}}{\ell!} e^{-n \omega_d r^{d}} \right)
	\end{align*}
	Differentiating in $r$ and summing a telescoping sum  yields
  $$f_{k,d}(r)  =  \frac{d n^{k} \omega_d^{k} r^{kd-1}}{(k-1)!}e^{-n\omega_d r^d}.$$
\end{proof}
The distance $r$ to the $k$th neighbor therefore has expectation  
  $$\int_0^\infty r f_{k,d}(r) dr =  \int_0^\infty \frac{d n^{k} \omega_d^{k} r^{kd}}{(k-1)!}e^{-n\omega_d r^d} dr= \frac{\Gamma \left(k+\frac{1}{d}\right) }{\omega_d^{1/d}(k-1)!} \frac{1}{n^{1/d}}$$
For example, in two dimensions, the expected distance to first five nearest neighbors is  
$$\frac{1}{2 \sqrt{n}},\frac{3}{4 \sqrt{n}},\frac{15}{16 \sqrt{n}},\frac{35}{32 \sqrt{n}},\frac{315}{256 \sqrt{n}}, \dots$$
respectively. We note, but do not prove or further pursue, that the sequence has some amusing properties and seems to be given (up to a factor of 2 in the denominator) by the series expansion 
$$ (1-x)^{\frac{3}{2}} = 1 + \frac{3}{2} x + \frac{15}{8}x^2 + \frac{35}{16} x^3 + \frac{315}{128} x^4 + \frac{693}{256} x^5 + \frac{3003}{1024} x^6 + \dots$$

\subsection{A Separation Lemma.} The proof of Theorem 1 shares a certain similarity with arguments that one usually encounters in subadditive Euclidean functional theory (we refer to the seminal work of Steele \cite{steel1, steel2}). The major difference is that our functional, the number of connected components, is scaling invariant, and slightly more troublesome, not monotone: adding a point can decrease the number of connected components. 
Suppose now that we are dealing with $n$ points and try to group them into $n/m$ sets of $m \ll n$ points each. Here, one should think of $m$ as a very large constant and $n \rightarrow \infty$. Ideally, we would like to argue
that the number of connected components among the $n$ is smaller than the sum of the connected components of each of the $n/m$ sets of $m$ points. This, however, need not be generally true (see Fig. \ref{fig:cut}).

\begin{center}
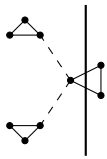
\begin{figure}[h!] 
\begin{tikzpicture}[scale=4]
\draw [thick] (-0.05,0.5) -- (-0.05,1);
\filldraw (-0.2, 0.9) circle (0.01cm);
\filldraw (-0.3, 0.9) circle (0.01cm);
\filldraw (-0.25, 0.95) circle (0.01cm);
\draw  (-0.2, 0.9) -- (-0.3, 0.9);
\draw  (-0.25, 0.95) -- (-0.3, 0.9);
\draw  (-0.25, 0.95) -- (-0.2, 0.9);

\filldraw (-0.2, 0.6) circle (0.01cm);
\filldraw (-0.3, 0.6) circle (0.01cm);
\filldraw (-0.25, 0.55) circle (0.01cm);
\draw  (-0.2, 0.6) -- (-0.3, 0.6);
\draw  (-0.25, 0.55) -- (-0.3, 0.6);
\draw  (-0.25, 0.55) -- (-0.2, 0.6);

\filldraw (-0.1, 0.75) circle (0.01cm);
\filldraw (0, 0.80) circle (0.01cm);
\filldraw (0, 0.7) circle (0.01cm);
\draw (-0.1, 0.75) -- (0,0.8);
\draw (-0.1, 0.75) -- (0,0.7);
\draw (0, 0.7) -- (0,0.8);
\draw [dashed] (-0.1, 0.75) -- (-0.2, 0.6);
\draw [dashed] (-0.1, 0.75) -- (-0.2, 0.9);
\end{tikzpicture}
\caption{Cutting this set of points in the middle leads to one point seeking two new nearest neighbors and decreases the total number of clusters.}
\label{fig:cut}
\end{figure}
\end{center}

\begin{lemma}[Separation Lemma.] \label{lem:separ} For every $\varepsilon > 0$, $k \in \mathbb{N}$ fixed, there exists $m \in \mathbb{N}$ such that, for all $N \in \mathbb{N}$ sufficiently large, we can subdivide $[0,1]^d$ into
$N/m$ sets of the same volume whose combined volume is at least $1 - \varepsilon$ such that the expected number of connected components of points following a Poisson distribution
of intensity $N$ in $[0,1]^d$ (each connecting to its $k-$nearest neighbors) is the sum of connected components of each piece with error $1 + \mathcal{O}(\varepsilon)$.
\end{lemma}

\begin{proof}  Recall that Lemma \ref{lem:unif} states that for all sets of points
$$ \sum_{x_i, x_j ~{\tiny \mbox{knn}}}{ \|x_i - x_j\|} \leq  c \cdot n^{\frac{d-1}{d}},$$
where the implicit constant $c$ depends on the dimension and $k$ but on nothing else. Let us now fix $\varepsilon > 0$ and see how to obtain such a decomposition. We start by decomposing the entire cube
into $N$ cubes of width $\sim N^{-1/d}$. This is followed by merging $m$ cubes in a cube-like fashion starting a corner. We then leave a strip of width $c \varepsilon^{-2}$ cubes in all directions and keep
constructing bigger cubes assembled out of $m$ smaller cubes and separated by $c \varepsilon^{-2}$ strips in this manner. We observe that $c \varepsilon^{-2}$ is a fixed constant: in particular, by making $m$ 
sufficiently large, the volume of the big cubes can be made to add up to $1 - \varepsilon^2$ of the total volume.
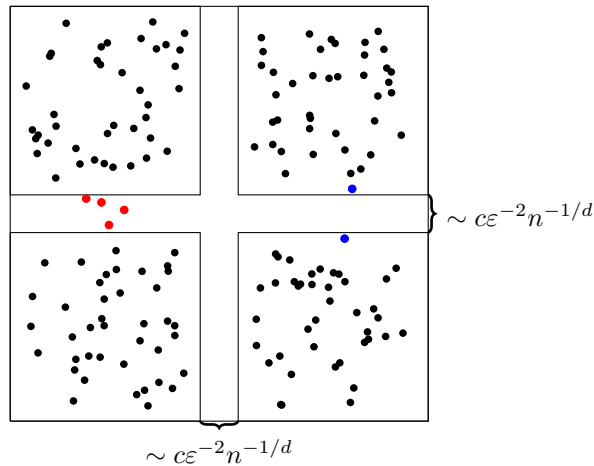
\begin{figure}[h!] \label{fig:sepa}
	\begin{tikzpicture}[scale=1]
	\draw (0,0) -- (5.5,0) -- (5.5,5.5) -- (0,5.5) -- (0,0);
	\draw (0,0) -- (2.5,0) -- (2.5,2.5) -- (0,2.5) -- (0,0); 
	\draw (0,5.5) -- (0,3) -- (2.5,3) -- (2.5,5.5) -- (0,5.5); 
	\draw (3,5.5) -- (3,3) -- (5.5,3) -- (5.5,5.5) -- (3,5.5); 
	\draw (3,0) -- (5.5,0) -- (5.5,2.5) -- (3,2.5) -- (3,0);
\draw [decoration={brace}, decorate,line width=1pt] 
	(5.5,3) -- node[right=3pt] {$\sim c \varepsilon^{-2} n^{-{1/d}}$}(5.5,2.5);
\draw [decoration={brace,mirror}, decorate,line width=1pt] 
	(2.5,0) -- node[below=3pt] {$\sim c \varepsilon^{-2} n^{-1/d}$}(3,0);
\coordinate (A) at (1.5,2.8);
\coordinate (B) at (1.3,2.6);
\coordinate (C) at (1.2,2.9);
\coordinate (D) at (1,2.95);
\filldraw [red,fill=red](A) circle (0.05cm);
\filldraw [red,fill=red](B) circle (0.05cm);
\filldraw [red,fill=red](C) circle (0.05cm);
\filldraw [red,fill=red](D) circle (0.05cm);

\filldraw [blue,fill=blue](4.5,3.08) circle (0.05cm);
\filldraw [blue,fill=blue](4.4,2.42) circle (0.05cm);
\pgfmathsetseed{3}
\foreach \p in {1,...,40}
{ 	\coordinate (penta-\p) at (1.05*rand+1.25,1.05*rand+1.25);
	\fill (penta-\p) circle (0.05);}
\foreach \p in {1,...,40}
{ 	\coordinate (penta-\p) at (1.05*rand+1.25,1.05*rand+1.25+3);
	\fill (penta-\p) circle (0.05);}

\foreach \p in {1,...,40}
{ 	\coordinate (penta-\p) at (1.05*rand+1.25+3,1.05*rand+1.25);
	\fill (penta-\p) circle (0.05);}
\foreach \p in {1,...,40}
{ \coordinate (penta-\p) at (1.05*rand+1.25+3,1.05*rand+1.25+3);
	\fill (penta-\p) circle (0.05);}
\end{tikzpicture}
\caption{The Separation Lemma illustrated for $m=4$.}
\end{figure}
A typical realization of $N$ will now have, if $m$ is sufficiently large, an arbitrarily small portion of points in the strips. These points may add or destroy clusters in the separate cube: in the worst
case, each single point is a connected component in itself (which, of course, cannot happen but suffices for this argument), which would change the total count by an arbitrarily small factor. Or, in the
other direction, these points, once disregarded, might lead to the separation of many connected components; each deleted edge can only create one additional component and each vertex has a uniformly
bounded number of edges, which leads to the same error estimate as above. However, there might also be a separation of edges that connected two points in big $m-$cubes. Certainly, appealing to the total
length, this number will satisfy
$$ \# \left\{ \mbox{edges connected different}~m-~\mbox{cubes} \right\} \leq  \varepsilon^2 \cdot n.$$
Since $\varepsilon$ was arbitrary small, the result follows.
\end{proof}

\subsection{Proof of Theorem 1}
\begin{proof} We first fix the notation: let $X_n$ denote the number of connected components of points drawn from a Poisson process with intensity $n$ in $[0,1]^d$ where each point is connected to its $k-$nearest
neighbors. The proof has several different steps. A rough summary is as follows.

\begin{enumerate}
\item We have already seen that $\varepsilon_{k,d} n \leq  \mathbb{E} X_n \leq
	n$. This implies that if the limit does not exist, then
	$(\mathbb{E}X_n)/n$ is sometimes large and sometimes small.  Corollary
	1 implies that if $(\mathbb{E}X_n)/n$ is small, then $(\mathbb{E}X_{n+m})/(n+m)$
	cannot be much larger as long as $m \ll n$. The scaling shows that $m$
	can actually be chosen to
grow linearly with $n$.  This means that whenever $(\mathbb{E}X_n)/n$ is small, we actually get an entire interval $[n, n + m]$ where that number is rather small and $m$ can grow linearly in $n$.
\item The next step is a decomposition of $[0,1]^d$ into many smaller cubes such that each set has an expected value of $n + m/2$ points. It is easy to see with standard bounds
that most sets will end up with $n+m/2 \pm \sqrt{n + m/2}$ points. Since $m$ can grow linearly with $n$, this is in the interval $[n, n+m]$ with likelihood close to 1.
\item The final step is to show that the sum of the clusters is pretty close to the sum of the clusters in each separate block (this is where the separation Lemma comes in). This then concludes
the argument and shows that for all sufficiently large number of points $N \gg n$, we end up having $\mathbb{E}X_N/N \sim \mathbb{E}X_n/n + \mbox{small error}.$ The small error decreases
for larger and larger values of $n$ and this will end up implying the result.
\end{enumerate}

\textbf{Step 1.} We have already seen that
$$  \varepsilon_{k,d} n \leq  \mathbb{E} X_n \leq n,$$
where the upper estimate is, of course, trivial. 
The main idea can be summarized as follows: if the statement were to fail, then both quantities
$$ \underline{a} := \liminf_{n \rightarrow \infty}{  \frac{ \mathbb{E} X_n}{n}}  \qquad \mbox{and} \qquad \overline{a} := \limsup_{n \rightarrow \infty}{  \frac{ \mathbb{E} X_n}{n}}$$
exist and are positive.
 This implies that we can find arbitrarily large numbers $n$ for which $\mathbb{E}X_n/n$ is quite small (i.e. close to $\underline{a}$). 
 We set, for convenience, $\delta := \overline{a} - \underline{a}$.
By definition, there exist arbitrarily large values of $n$ such that
$$ \mathbb{E} X_n \leq \left( \underline{a} + \frac{\delta}{10} \right)n.$$
 It follows then, from Corollary \ref{cor1}, that
$$ \mathbb{E} X_{n+m} \leq \left( \underline{a} + \frac{\delta}{10} \right)n + c_d  k m \leq   \left( \underline{a} + \frac{\delta}{5} \right)(n+m) \qquad
\mbox{for all} \qquad m \leq \frac{\delta n}{10 c_d k}=: m_0.$$
This means that for all bad values $n$, all the values $n+m$ with $m \lesssim_{d,k, \delta} n$ are still guaranteed to be very far away from achieving any value close to $\overline{a}$. 
We also note explicitly that the value of $m$ can be chosen as a fixed proportion of $n$ independently of the size of $n$, i.e. $m$ is growing lineary with $n$.\\

\textbf{Step 2.}
Let us now consider a Poisson distribution $P_{\lambda}$ with intensity $\lambda$ being given by $ \lambda = n + m_0/2.$ It is easy to see that, for every
$\varepsilon > 0$ and all $n$ sufficiently large (depending only on $\varepsilon$)
$$ \mathbb{P}(n \leq P_{\lambda} \leq n+m_0) = e^{-\lambda} \sum_{i=n}^{n+m_0}{ \frac{\lambda^i}{i!}} \geq 1 - \varepsilon.$$
This follows immediately from the fact that the variance is $\lambda$ and the
classical Tschebyscheff inequality arguments. Indeed, much stronger results are
true since the scaling of a standard deviation is at the square root of the
sample size and we have an interval growing linearly in the sample size --
moreover, there are large deviation tail bounds, so one could really show
massively stronger quantitative results but these are not required here. We now
set $\varepsilon = \delta/100$ and henceforth only consider values of $n$ that
are so large that the above inequality holds.\\

\textbf{Step 3.}  When dealing with a Poisson process of intensity $N \gg n$, we can decompose, using the Separation Lemma, the unit cube $[0,1]^d$ into disjoint, separated cubes with volume $\sim n/N$ with a volume error of size $ \sim N^{\frac{1}{d}} \ll N$ (due to not enough little cubes fitting exactly). When considering the effect of this Poisson process inside a small cube, we see that with very high likelihood ($1-\varepsilon$), the number of points is in the interval
$[n, n + m_0]$. The Separation Lemma moreover guarantees that the number of points that end up between the little cubes (`fall between the cracks') is as small a proportion of $N$ as we wish provided
$n$ is sufficiently large. Let us now assume that the total number of connected components among the $N$ points is exactly the same as the sum of the connected components in the little cubes. Then
we would get that
\begin{align*}
 \mathbb{E} X_N \leq  \left(\underline{a} + \frac{\delta}{5}\right)N.
\end{align*}
This is not entirely accurate: there are $\varepsilon_2 N$ points that fell between the cracks (with $\varepsilon_2$ sufficiently small if $n$ is sufficiently large) and there are $\varepsilon N$ points that end
up in cubes that have a total number of points outside the $[n, n+m_0]$ regime. However, any single point may only add $c_d k$ new clusters and thus
\begin{align*}
 \mathbb{E} X_N \leq  \left(\underline{a} + \frac{\delta}{5}\right)N + c_d k \left(\varepsilon +\varepsilon_2\right)N
\end{align*}
and by making $\varepsilon+\varepsilon_2 \leq \delta/100$ (possibly by increasing $n$), we obtain
$$\overline{a} \leq \underline{a} + \frac{2 \delta}{5},$$
which is a contradiction, since $\delta = \overline{a} - \underline{a}$.
\end{proof}

\section{Proof of Theorem 2}

\subsection{The Erd\H{o}s-Renyi Lemma.} Before embarking on the proof, we describe a short statement. The proof
is not subtle and follows along completely classical lines but occurs in an unfamiliar regime: we are interested in ensuring that the likelihood of obtaining a disconnected
graph is very small. The subsequent
argument, which is not new but not easy to immediately spot in the literature, is included for the convenience of the reader (much stronger and more subtle results usually 
focus on the threshold $p = (1 \pm \varepsilon) (\log{n})/n$).
\begin{lemma}  \label{lem:erdren} Let $G(n,p)$ an Erd\H{o}s-Renyi graph with $p > 10\log{n}/n$
 Then, for $n$ sufficiently large,
$$ \mathbb{P}\left( G(n,p)~\mbox{is disconnected}\right) \lesssim e^{-pn/3}.$$
\end{lemma}
\begin{proof} It suffices to bound the likelihood of finding an isolated set of $k$ vertices from above, where $1 \leq k \leq n/2$. For any fixed set
of $k$ vertices of it being isolated is bounded from above by
$$ \mathbb{P}\left(\mbox{fixed set of}~k~\mbox{vertices being disconnected}\right) \leq (1-p)^{k(n-k)}$$
and thus, using the union bound,
$$ \mathbb{P}\left( G(n,p)~\mbox{is connected}\right) \leq  \sum_{k=1}^{n/2}{  \binom{n}{k}(1-p)^{k(n-k)}}.$$
We use
$$ \binom{n}{k} \leq \left( \frac{ n e }{k}\right)^k$$
to rewrite the expression as 
\begin{align*}
 \sum_{k=1}^{n/2}{  \binom{n}{k}(1-p)^{k(n-k)}} &\leq \sum_{k=1}^{n/2}{ e^{k + k \log{n} + k \log{k} + \left[\log{(1-p)}\right] k (n-k)}} \\
&\leq  \sum_{k=1}^{n/2}{ e^{k \left(3  \log{n} + \left[\log{(1-p)}\right]  (n-k)\right) }} \\
&\leq  \sum_{k=1}^{n/2}{ e^{k \left(3  \log{n} + \left[\log{(1-p)}\right]  (n/2)\right) }}\\
&\lesssim e ^{3 \log{n} + \left[\log{(1-p)}\right] n/2},
\end{align*}
where the last step is merely the summation of a geometric series and valid as soon as
$$ 3  \log{n} + \left[\log{(1-p)}\right] \frac{n}{2} <  0,$$
which is eventually true for $n$ sufficiently large since $p > 10\log{n}/n$.
\end{proof}

\subsection{A Mini-Percolation Lemma}
The purpose of this section is to derive rough bounds for a percolation-type problem.
\begin{lemma} \label{lem:perc}
Suppose we are given a  grid graph on $\left\{1,2,\dots,n\right\}^d$ and remove each of the $n^d$ points with likelihood $p = (\log{n})^{-c}$ for some $c>0$. Then, for $n$ sufficiently large, 
there is a giant component with expected size $n^d - o(n^d)$.
\end{lemma}
The problem seems so natural that some version of it must surely be known. It seems to be dual to classical percolation problems (in the sense that one randomly deletes vertices instead
of edges). It is tempting to believe that the statement remains valid for $p$ all the way up to some critical exponent $0 < p_{crit} < 1$ that depends on the dimension (and grows as the dimension gets larger).
Before embarking on a proof, we show a separate result. We will call a subset $A \subset \left\{1, 2, \dots, n \right\}^d$ \textit{connected} if the resulting graph is connected: here, edges are given by
connecting every node to all of its adjacent nodes that differ by at most one in each coordinate (that number is bounded from above by $3^{d}-1$).

\begin{lemma} The number of connected components $A$ in the grid graph over $\left\{1,2,\dots,n\right\}^d$ with cardinality $|A| = \ell$ is bounded from above
$$ \# \mbox{number of connected components of size}~\ell  \leq n^d \left(2^{3^d -1}\right)^{\ell}.$$
\end{lemma}
\begin{proof} The proof proceeds in a fairly standard way by constructing a combinatorial encoding. We show how this is done in two dimensions, giving an upper bound of $n^2 256^{\ell}$ -- the
construction immediately transfers to higher dimensions in the obvious way. 
\begin{center}
\begin{figure}[h!] \label{fig:perc}
\begin{tikzpicture}[scale=3]
\draw [ultra thick] (0,0) -- (0,1) -- (1,1) -- (1,0) -- (0,0);
\draw [ultra thick] (0,0.33) -- (1,0.33);
\draw [ultra thick] (0,0.66) -- (1,0.66);
\draw [ultra thick] (0.33,0) -- (0.33,1);
\draw [ultra thick] (0.66,0) -- (0.66,1);
\node at (0.15, 0.15) {7};
\node at (0.15, 0.5) {8};
\node at (0.15, 0.85) {1};
\node at (0.5, 0.85) {2};
\node at (0.85, 0.85) {3};
\node at (0.85, 0.5) {4};
\node at (0.85, 0.15) {5};
\node at (0.5, 0.15) {6};
\end{tikzpicture}
\caption{Fixing labels for the 8 immediate neighbors.}
\end{figure}
\end{center}
The encoding is given by a direct algorithm.
\begin{enumerate}
\item Pick an initial vertex $x_0 \in A$. Describe which of the 8 adjacent squares are occupied by picking a subset of $\left\{1,2,\dots, 8\right\}$.
\item Implement a depth-first search as follows: pick the smallest number in the set attached to $x_0$ and describe its neighbors, if any, that
are distinct from previously selected nodes as a subset of $\left\{1,2,\dots, 8\right\}$.
\item Repeat until all existing neighbors have been mapped out (the attached set is the empty set) and then go back and describe the next branch. 
\end{enumerate}
Just for clarification, we quickly show the algorithm in practice. Suppose we are given an initial point $x_0$ and the sequence of sets
$$ \left\{4,5\right\}, \left\{3,4\right\}, \left\{\right\}, \left\{\right\}, \left\{5\right\}, \left\{4\right\}, \left\{\right\},$$
then this uniquely identifies the set showing in Figure \ref{fig:reconstruct}.

\begin{center}
\begin{figure}[h!] 
\begin{tikzpicture}[scale=0.4]
\draw[fill=gray!60] (17+7,0) rectangle (18+7,1);
\draw[fill=gray!60] (17,0) rectangle (18,1);
\draw[fill=gray!60] (17-7,0) rectangle (18-7,1);
\draw[fill=gray!60] (17-14,0) rectangle (18-14,1);
\draw[fill=gray!60] (16,0) rectangle (17,1);
\draw[fill=gray!60] (16-7,0) rectangle (17-7,1);
\draw[fill=gray!60] (16-14,0) rectangle (17-14,1);
\draw[fill=gray!60] (16-7,2) rectangle (17-7,3);
\draw[fill=gray!60] (16-7,3) rectangle (17-7,4);
\draw[fill=gray!60] (16-14,2) rectangle (17-14,3);
\draw[fill=gray!60] (16-14,3) rectangle (17-14,4);
\draw[fill=gray!60]  (1,1) rectangle (2,2);
\draw[fill=gray!60]  (1,2) rectangle (2,3);
\draw[step=1cm, ultra thick] (0,0) grid (5,5);
\draw[step=1cm, ultra thick] (7,0) grid (12,5);
\draw[step=1cm, ultra thick] (14,0) grid (19,5);
\draw[step=1cm, ultra thick] (21,0) grid (26,5);
\draw[step=1cm, ultra thick] (28,0) grid (33,5);
\draw[fill=black] (0,2) rectangle (1,3);
\draw[fill=black] (7,2) rectangle (8,3);
\draw[fill=black] (14,2) rectangle (15,3);
\draw[fill=black] (21,2) rectangle (22,3);
\draw[fill=black] (28,2) rectangle (29,3);
\draw[fill=black] (8,2) rectangle (9,3);
\draw[fill=black] (8,1) rectangle (9,2);
\draw[fill=black] (15,2) rectangle (16,3);
\draw[fill=black] (15,1) rectangle (16,2);
\draw[fill=black] (22,2) rectangle (23,3);
\draw[fill=black] (22,1) rectangle (23,2);
\draw[fill=black] (29,2) rectangle (30,3);
\draw[fill=black] (29,1) rectangle (30,2);
\draw[fill=black] (16,2) rectangle (17,3);
\draw[fill=black] (16,3) rectangle (17,4);
\draw[fill=black] (16+7,2) rectangle (17+7,3);
\draw[fill=black] (16+7,3) rectangle (17+7,4);
\draw[fill=black] (16+14,2) rectangle (17+14,3);
\draw[fill=black] (16+14,3) rectangle (17+14,4);
\draw[fill=black] (16+7,0) rectangle (17+7,1);
\draw[fill=black] (16+14,0) rectangle (17+14,1);
\draw[fill=black] (17+14,0) rectangle (18+14,1);
\end{tikzpicture}
\caption{The starting point followed by generating the connected component described by the sequence of sets $ \left\{4,5\right\}, \left\{3,4\right\}, \left\{\right\}, \left\{\right\}, \left\{5\right\}, \left\{4\right\}, \left\{\right\}$. }
\label{fig:reconstruct}
\end{figure}
\end{center}
Clearly, this description returns $\ell$ subsets of $\left\{1,\dots, 8\right\}$ of which there are 256. Every element in $A$ generates exactly one such subset and every connected component
can thus be described by giving the $n^d$ initial points and then a list of $\ell$ subsets of $\left\{1,\dots, 8\right\}$. This implies the desired statement; we note that the actual number should
be much smaller since this way of describing connected components has massive amounts of redundancy and overcounting. 
\end{proof}

\begin{proof}[Proof of Lemma \ref{lem:perc}] The proof is actually fairly lossy and proceeds by massive overcounting. The only way to remove mass from the giant block is to remove points in an organized manner: adjacent squares have to be removed in a way that encloses a number of squares that are not removed (see Fig. \ref{fig:perc}).

\begin{center}
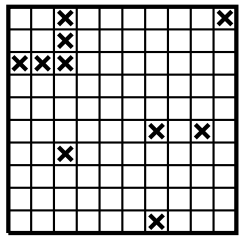
\begin{figure}[h!] \label{fig:perc}
\begin{tikzpicture}[scale=3]
\draw [ultra thick] (0,0) -- (0,1) -- (1,1) -- (1,0) -- (0,0);
\draw [thick] (0,0.1) -- (1,0.1);
\draw [thick] (0,0.2) -- (1,0.2);
\draw [thick] (0,0.3) -- (1,0.3);
\draw [thick] (0,0.4) -- (1,0.4);
\draw [thick] (0,0.5) -- (1,0.5);
\draw [thick] (0,0.6) -- (1,0.6);
\draw [thick] (0,0.7) -- (1,0.7);
\draw [thick] (0,0.8) -- (1,0.8);
\draw [thick] (0,0.9) -- (1,0.9);
\draw [thick] (0.1,0) -- (0.1,1);
\draw [thick] (0.2,0) -- (0.2,1);
\draw [thick] (0.3,0) -- (0.3,1);
\draw [thick] (0.4,0) -- (0.4,1);
\draw [thick] (0.5,0) -- (0.5,1);
\draw [thick] (0.6,0) -- (0.6,1);
\draw [thick] (0.7,0) -- (0.7,1);
\draw [thick] (0.8,0) -- (0.8,1);
\draw [thick] (0.9,0) -- (0.9,1);
\draw [ultra thick] (0.92, 0.92) -- (0.98, 0.98);
\draw [ultra thick] (0.98, 0.92) -- (0.92, 0.98);
\draw [ultra thick] (0.22, 0.92) -- (0.28, 0.98);
\draw [ultra thick] (0.22, 0.98) -- (0.28, 0.92);
\draw [ultra thick] (0.22, 0.82) -- (0.28, 0.88);
\draw [ultra thick] (0.22, 0.88) -- (0.28, 0.82);
\draw [ultra thick] (0.22, 0.72) -- (0.28, 0.78);
\draw [ultra thick] (0.22, 0.78) -- (0.28, 0.72);
\draw [ultra thick] (0.12, 0.72) -- (0.18, 0.78);
\draw [ultra thick] (0.12, 0.78) -- (0.18, 0.72);
\draw [ultra thick] (0.02, 0.72) -- (0.08, 0.78);
\draw [ultra thick] (0.02, 0.78) -- (0.08, 0.72);
\draw [ultra thick] (0.62, 0.42) -- (0.68, 0.48);
\draw [ultra thick] (0.62, 0.48) -- (0.68, 0.42);
\draw [ultra thick] (0.22, 0.32) -- (0.28, 0.38);
\draw [ultra thick] (0.22, 0.38) -- (0.28, 0.32);
\draw [ultra thick] (0.82, 0.42) -- (0.88, 0.48);
\draw [ultra thick] (0.82, 0.48) -- (0.88, 0.42);
\draw [ultra thick] (0.62, 0.02) -- (0.68, 0.08);
\draw [ultra thick] (0.62, 0.08) -- (0.68, 0.02);
\end{tikzpicture}
\caption{Removing tiny squares randomly: this random sample ends up removing a bit more from the giant component but is quite unlikely.}
\end{figure}
\end{center}
The next question is how many other points can possibly be captured by a connected component on $\ell-$nodes. The isoperimetric principle implies
$$ \# \mbox{blocks captured by}~\ell~\mbox{nodes} \lesssim_d \ell^{\frac{d}{d-1}} \leq \ell^2.$$
Altogether, this implies we expect to capture at most
\begin{align*}
\sum_{\ell=1}^{n^d} n^d \left(2^{3^d -1}\right)^{\ell} \left( \log{n}\right)^{-c \ell}  \ell^2 \leq n^d  \sum_{\ell=1}^{\infty}\left( \frac{ 2^{3^d -1}}{ (\log{n})^{c}}\right)^{\ell}   \ell^2 \lesssim  \frac{ n^d}{(\log{n})^c},
\end{align*}
where the last inequality holds as soon as $\log{n}^c \gg 2^{3^d-1}$ and follows from the derivative geometric series
$$ \sum_{\ell=1}^{\infty}{ \ell^2 q^{\ell} } = \frac{q(1+q)}{(1-q)^3} \qquad \mbox{whenever}~|q| < 1.$$
\end{proof}

\textit{Remark.} There are two spots where the argument is fairly lossy. First of all, every connected component on $\ell$ nodes is, generically, counted as $\ell$ connected components of length
$\ell -1$, as $\sim \ell^2$ connected components of size $\ell - 2$ and so on. The second part of the argument is the application of the isoperimetric inequality: a generic connected component
on $\ell$ nodes will capture $ \ll \ell^{2}$ other nodes. These problems seem incredibly close to existing research and it seems likely that they either have been answered already or that techniques
from percolation theory might provide rather immediate improvements.

\subsection{Outline of the Proof} 
The proof proceeds in three steps.

\begin{enumerate}
\item Partition the unit cube into smaller cubes such that each small cube has an expected number of $\sim \log{n}$ points (and thus, the number of cubes is $\sim n/\log{n}$). Show that the likelihood of a single cube containing significantly more or significantly less points is small.
\item Show that graphs within the cube are connected with high probability.
\item Show that there are connections between the cubes that ensure connectivity.
\end{enumerate}

\subsection{Step 1.}
We start by partitioning $[0,1]^d$ in the canonical manner into axis-parallel cubes having side-length $\sim \left(c\log{n}/n\right)^{1/d}$ for some constant $c$ to be chosen later. There are
roughly $\sim n/(c \log{n})$ cubes and they have measure $\sim c \log{(n)}/n$. 
We start by bounding the likelihood of a one such cube containing $\leq \log{n}/100$ points.
Clearly, this likelihood can be written as a Bernoulli random variables 
$$ \mbox{number of points in cube} = \mathcal{B}\left(n, \frac{c\log{n}}{n}\right).$$
The Chernoff-Hoeffding theorem \cite{hoeffding}  implies
$$ \mathbb{P}\left( \mathcal{B}\left(n, \frac{c\log{n}}{n}\right) \leq \frac{\log{n}}{100} \right) \leq \exp\left( - n D\left(\frac{\log{n}}{100n} || \frac{c \log{n}}{n} \right)\right),$$
where $D$ is the relative entropy
$$ D(a || b) = a \log{\frac{a}{b}} + (1-a) \log{\frac{1-a}{1-b} }.$$
Here, we have, for $n$ large,
$$D\left(\frac{\log{n}}{100n} || \frac{c \log{n}}{n} \right) \sim \frac{\log{n}}{n}\left(c - \frac{1}{100} + \frac{1}{100}\log{\frac{1}{100c}}\right).$$
This implies that for $c$ sufficiently large, we have
\begin{align*}
\mathbb{P}\left(\mbox{fixed cube has less than}~\frac{\log{n}}{100}~\mbox{points}\right) \lesssim_{c, \varepsilon} \frac{1}{n^{c-\varepsilon}}
\end{align*}
and the union bound implies
$$ \mathbb{P}\left(\mbox{there exists cube that has less than}~\frac{\log{n}}{100}~\mbox{points}\right) \lesssim_{c, \varepsilon} \frac{1}{n^{c - 1 -\varepsilon}}$$
The same argument also shows that
$$
\mathbb{P}\left(\mbox{exists cube with more than}~10c \log{n}~\mbox{points}\right) \lesssim \frac{1}{n^{c}}.
$$
This means we have established the existence of a constant $c$ such that with likelihood tending to 1 as $n \rightarrow \infty$ (at arbitrary inverse polynomial speed provided
$c$ is big enough)
$$ \forall ~\mbox{cubes}~Q  \qquad \qquad \frac{\log{n}}{100} \leq  \#\left\{\mbox{points in}~Q\right\} \leq 10 c \log{n}.$$
We henceforth only deal with cases where these inequalities are satisfied for all cubes.

\subsection{Step 2.}  
We now study what happens within a fixed cube $Q$. The cube is surrounded by at most $3^d-1$ other cubes each of which contains at most
$10c \log{n}$ points. This means that if, for any $x \in Q$, we compile a list of its $3^d 10 c \log{n}$ nearest neighbours, we are guaranteed
that every other element in $Q$ is on that list. Let us suppose that the rule is that each point is connected to each of its $3^d 10 c \log{n}-$nearest
neighbors with likelihood
$$ p = \frac{m}{3^d 10 c \log{n}}.$$
Then, Lemma \ref{lem:erdren} implies that for $m \gtrsim 10 \log{\left(3^d 10 c \log{(n)}\right)} \sim_{d,c} \log \log{n}$ the likelihood of obtaining a connected 
graph strictly within $Q$ is at least $(\log{n})^{-c}$. Lemma \ref{lem:perc} then implies the result provided we can ensure that points in cubes
connect to their neighboring cubes.

\subsection{Step 3.} We now establish that the likelihood of a cube $Q$ having, for every adjacent cube $R$, a point that connects to a point in $R$
is large. The adjacent cube has $\sim \log{n}$ points. The likelihood of a fixed point in $Q$ not connecting to any point in $R$ is
$$ \leq  \left( 1 - \frac{\frac{\log{n}}{100}}{ 3^d 10c \log{n}} \right)^{c\log{\log{n}}}   = \left( 1 - \frac{1}{3^d 1000c} \right)^{c\log{\log{n}}} \lesssim  \left(\log{n}\right)^{-\varepsilon_{c,d}}.$$
The likelihood that this is indeed true for every point is then bounded from above by
$$  \left(\log{n}\right)^{-\varepsilon_{c,d} \log{n}/100}  \lesssim n^{-1},$$
which means, appealing again to the union bound, that this event occurs with a likelihood going to 0 as $n \rightarrow \infty$. $\qed$\\

\textbf{Connectedness.} It is not difficult to see that this graph is unlikely to be connected. For a fixed vertex $v$, there are $\sim c \log{n}$
possible other vertices it could connect to and $\sim c \log{n}$ other vertices might possibly connect to $v$. Thus
$$ \mathbb{P}\left(v~\mbox{is isolated}\right) \lesssim \left(1 - \frac{c_2 \log{\log{n}}}{\log{n}}\right)^{c_3 \log{n}} \leq e^{-c_2 c_3 \log{\log{n}}} = \frac{1}{(\log{n})^{c_2 c_3}}.$$
This shows that we can expect at least $n \left(\log{n}\right)^{-c_2 c_3}$ isolated vertices. This also shows that the main obstruction to connectedness
is the nontrivial likelihood of vertices not forming edges to other vertices. This suggests a possible variation of the graph construction that is discussed
in the next section.

\section{An Ulam-type modification}

There is an interesting precursor to the Erd\"os-Renyi graph that traces back to a question of Stanislaw Ulam in the \textit{Scottish Book}.
\begin{quote}
\textbf{Problem 38: Ulam.}  
Let there be given $N$ elements (persons). To each element we attach $k$ others
among the given $N$ at random (these are friends of a given person). What is the
probability $\mathbb{P}_k(N)$ that from every element one can get to every other element through
a chain of mutual friends? (The relation of friendship is not necessarily symmetric!)
Find $\lim_{N\rightarrow \infty} \mathbb{P}_{k}(N)$ (0 or 1?). (Scottish Book, \cite{scot})
\end{quote}

 We quickly establish a basic variant of the Ulam-type question sketched in the introduction since the argument
itself is rather elementary. It is a natural variation on the Ulam question (friendship now being symmetric) and the usual Erd\"os-Renyi
argument applies. A harder problem (start by constructing a directed graph, every vertex forms an outgoing edge to $k$ other randomly
chosen vertices, and then construct an undirected graph by including edges where both $uv$ and $vu$ are in the directed graph)
was solved by Jungreis \cite{scot}.
\begin{quote}
\textbf{Question.} If we are given $n$ randomly chosen points in $[0,1]^d$ and connect each vertex to exactly $c_1 \log{\log{n}}$ of its
$c_2 \log{n}$ nearest neighbors, is the arising graph connected with high probability?
\end{quote}

We have the following basic Lemma that improves on the tendency of Erd\H{o}s-Renyi graphs to form
small disconnected components.

\begin{lemma} \label{lem:ulam} If we connect each of $n$ vertices to exactly $k$ other randomly chosen vertices, then
$$ \mathbb{P}\left(\mbox{Graph is disconnected}\right) \lesssim_k \frac{1}{n^{(k-1)(k+1)}}.$$
\end{lemma}

\begin{proof} If the graph is disconnected, then we can find
a connected component of $\ell \leq n/2$ points that is not connected to the remaining $n-\ell$ points. For a fixed set of $\ell \geq k+1$ points, the likelihood of
this occurring is
$$ \mathbb{P}(\mbox{fixed set of}~\ell~\mbox{points is disconnected from the rest}) \leq \left(\frac{\ell}{n}\right)^{\ell k} \left(\frac{n-\ell}{n}\right)^{\ell (n-\ell)}.$$
An application of the union bound shows that the likelihood of a graph being disconnected can be bounded from above by
\begin{align*}
\sum_{\ell=k+1}^{n/2} \binom{n}{\ell} \left( \frac{\ell}{n} \right)^{k \ell} \left( \frac{n-\ell}{n} \right)^{2(n-\ell)}  &\leq \sum_{\ell=k+1}^{n/2} \left( \frac{ne}{\ell}\right)^{\ell} \left( \frac{\ell}{n} \right)^{k \ell} \left( \frac{n-\ell}{n} \right)^{k(n-\ell)} \\
&\lesssim  \sum_{\ell=k+1}^{n/2} e^{\ell} \left( \frac{\ell}{n} \right)^{(k-1) \ell} \left( \frac{n-\ell}{n} \right)^{k(n-\ell)} 
\end{align*}
We use that the approximation to $e$ converges from below
$$ \left(1 - \frac{k}{n}\right)^{n} \leq e^{-k}$$
to bound
\begin{align*}
 \sum_{\ell=k+1}^{n/2} e^{\ell} \left( \frac{\ell}{n} \right)^{(k-1) \ell} \left( \frac{n-\ell}{n} \right)^{k(n-\ell)}  &\leq  \sum_{\ell=k+1}^{n/2} e^{\ell} \left( \frac{\ell}{n} \right)^{(k-1) \ell} e^{-k \ell} \left( \frac{n-\ell}{n} \right)^{-k\ell}  \\
&\leq e \sum_{\ell=k+1}^{n/2}  \left( \frac{\ell}{n} \right)^{(k-1) \ell}  \left( \frac{n}{n-\ell} \right)^{k\ell}.
\end{align*}
We use once more the approximation to Euler's number to argue
$$ \left( \frac{n}{n-\ell} \right)^{k\ell} = \left(1 + \frac{\ell}{n-\ell}\right)^{(n-\ell) \frac{k \ell}{n-\ell}} \leq \exp\left(\frac{k \ell^2}{n-\ell}\right).$$
This expression is $\sim 1$ as long as $\ell \lesssim \sqrt{n}$. This suggests splitting the sum as 
$$ \sum_{\ell=k+1}^{n/2}  \left( \frac{\ell}{n} \right)^{(k-1) \ell}  \left( \frac{n}{n-\ell} \right)^{k\ell} \lesssim \sum_{\ell=k+1}^{\sqrt{n}}  \left( \frac{\ell}{n} \right)^{(k-1) \ell} +
\sum_{\ell=\sqrt{n}}^{n/2}  \left( \frac{\ell}{n} \right)^{(k-1) \ell}  \left( \frac{n}{n-\ell} \right)^{k\ell}.$$
We start by analyzing the first sum. We observe that the first term yields exactly the desired asymptotics. We shall show that the remainder of the sum is small by
comparing ratios of consecutive terms
$$ \frac{ \left( \frac{\ell +1}{n}\right)^{(k-1)(\ell+1)}}{  \left( \frac{\ell }{n}\right)^{(k-1)(\ell)}} = \frac{1}{n^{k-1}} \left(1 + \frac{1}{\ell}\right)^{k-1} \leq \frac{1}{n^{k-1}} \left(1 + \frac{1}{k+1}\right)^{k-1} \leq \frac{e}{n^{k-1}} \ll 1.$$
This implies that we can dominate that sum by a geometric series which itself is dominated by its first term. The same argument will now be applied to the second sum. We observe that the same ratio-computation shows
$$ \sum_{\ell=\sqrt{n}}^{n/2}  \left( \frac{\ell}{n} \right)^{(k-1) \ell}  \left( \frac{n}{n-\ell} \right)^{k\ell} \leq \sum_{\ell=\sqrt{n}}^{n/2}  \left( \frac{\ell}{n} \right)^{(k-1) \ell} 2^{k\ell} \lesssim  \left(\frac{1}{\sqrt{n}}\right)^{(k-1)\sqrt{n}} 2^{k \sqrt{n}} \ll \frac{1}{n^{(k-1)(k+1)}}.$$
\end{proof}

\vspace{30pt}

\textbf{Acknowledgement.}  This work was supported by NIH grant 1R01HG008383-01A1 (to GCL and YK), NIH MSTP Training Grant T32GM007205 (to GCL), United States-Israel Binational Science Foundation and the United States National Science Foundation grant no. 2015582 (to GM).

\end{document}